\documentclass[11pt]{amsart}
\usepackage{amsmath,amsfonts,amssymb,amsxtra,amscd,enumerate,amsthm}
\usepackage{color}
\usepackage{latexsym}
\usepackage{epsfig}
  \usepackage{fullpage}

\newcommand\les{\lesssim}

\newcommand\bj{\textbf{j}}

\newcommand\R{\mathbb{R}}

\newcommand\C{\mathbb{C}}
\newcommand\Z{\mathbb{Z}}
\newcommand\N{\mathbb{N}}

\newcommand{\calE}{\mathcal E}

\newcommand{\calH}{\mathcal H}
\newcommand{\calR}{\mathcal R}

\newcommand{\calC}{\mathcal C}

\newcommand{\ls}{{\lesssim}}

\newcommand\la{\langle}
\newcommand\ra{\rangle}
\newtheorem{theo}{Theorem}
\numberwithin{theo}{section} 
\newtheorem{lema}[theo]{Lemma}

\newtheorem{corol}[theo]{Corollary}
\newtheorem{defin}[theo]{Definition}

\newtheorem{con}{Condition}

\numberwithin{equation}{section}

\begin{document}
\title{The multilinear restriction estimate: almost optimality and localization}

\author[I. Bejenaru]{Ioan Bejenaru} \address{Department
  of Mathematics, University of California, San Diego, La Jolla, CA
  92093-0112 USA} \email{ibejenaru@math.ucsd.edu}

\begin{abstract} The first result in this paper provides a very general $\epsilon$-removal argument for the multilinear restriction estimate. 
The second result provides a refinement of the multilinear restriction estimate in the case when some terms have appropriate localization properties;
this generalizes a prior result of the author in \cite{Be1}.

\end{abstract}


\maketitle

\section{Introduction}

For $n \geq 2$, let $U \subset \R^{n-1}$ be an open, bounded and connected  neighborhood of the origin and let $\Sigma: U \rightarrow \R^{n}$ be a smooth parametrization of an $n-1$-dimensional submanifold of $\R^{n}$ (hypersurface), which we denote by $S=\Sigma(U)$. By a smooth parametrization we mean that $\Sigma$ satisfies
\begin{equation} \label{smooth}
\| \partial^\alpha \Sigma \|_{L^\infty(U)} \lesssim_\alpha 1,
\end{equation} 
for $|\alpha| \leq N$ for some large $N$. We say that $S$ is a smooth hypersurface if it admits a parametrization satisfying \eqref{smooth}.
To such a parametrization of $S$ we associate the extension operator $\calE$  defined by
\[
\calE f(x) = \int_U e^{i x \cdot \Sigma(\xi)} f(\xi) d\xi, 
\]
where $f \in L^1(U)$ and $x \in \R^n$. This operator is closely related to a more intrinsic formulation of the extension operator:
\[
\tilde \calE g(x)= \int_S e^{ix \cdot \omega} g(\omega) d \sigma_S(\omega),
\]
where $g \in L^1(S; d \sigma_S)$ and $x \in \R^n$. Indeed, using the parametrization $\Sigma$ as above we obtain
$\omega=\Sigma(\xi)$ and $d \sigma_S(\omega)=|\frac{\partial \Sigma}{\partial_{\xi_1}} \wedge ... \wedge \frac{\partial \Sigma}{\partial_{\xi_{n-1}}}| d \xi$, thus
with $f(\xi)= g(\Sigma(\xi)) |\frac{\partial \Sigma}{\partial_{\xi_1}} \wedge ... \wedge \frac{\partial \Sigma}{\partial_{\xi_{n-1}}}|$, the two formulations are equivalent. 
We will use good parameterizations in the sense that $ |\frac{\partial \Sigma}{\partial_{\xi_1}} \wedge ... \wedge \frac{\partial \Sigma}{\partial_{\xi_{n-1}}}| \approx 1$
throughout the domain and, given that all results are in terms of $L^p$ norms, the equivalence is carried out at the levels of results as well. 

In order to avoid unnecessary technical issues, throughout this paper we assume that, in the definition of $\calE f$, the support of $f$ is a compact subset of $U$; in other words $f$ is supported away from the boundary of $U$. A similar assumption will be in place for the corresponding $g$ that appears in the definition of 
$\tilde \calE g$.

Given $k$ smooth, compact hypersurfaces $S_i \subset \R^{n}, i=1,..,k$, where $1 \leq k \leq n$, we consider the following $k$-linear restriction estimate is the following
inequality 
\begin{equation}  \label{MRE}
\| \Pi_{i=1}^k \calE_i f_i \|_{L^p(\R^{n})} \les \Pi_{i=1}^k \| f_i \|_{L^2(U_i)}.  
\end{equation}
In a more compact format this estimate is abbreviated as follows:
\begin{equation} \label{Rsp}
\calR^*(2 \times ... \times 2 \rightarrow p).
\end{equation}
A natural condition to impose on the hypersurfaces is the standard transversality condition: there exists $\nu >0$ such that
\begin{equation} \label{trans}
| N_1(\zeta_1) \wedge ... \wedge N_{k}(\zeta_k) | \geq \nu,
\end{equation}
for any $\zeta_i \in S_i, i \in \{1,..,k\}$; here $N_i(\zeta_i)$ is the unit normal at $\zeta_i$ to $S_i$ and it is clear that the choice of orientation is not important. 

There are two main conjectures regarding the optimal exponent in \eqref{Rsp}. For the generic case, when only the transversality condition
is assumed, the optimal conjectured exponent for \eqref{Rsp} is $p=\frac2{k-1}$. 
There is also the non-generic case when, in addition to the transversality condition, one assumes also appropriate curvature conditions on the hypersurfaces
$S_i$ and in this case the optimal conjectured exponent in \eqref{Rsp} is $p=\frac{2(n+k)}{k(n+k-2)}$ (note that this value is $< \frac2{k-1}$). 
This paper considers problems related to the generic case, but the results obtained here will have implications for the non-generic case. 

In \cite{BeCaTa} Bennett, Carbery and Tao established the near-optimal version of the conjectured result in the generic case: for any $\epsilon > 0$
and for any $R >0$ and $x_0 \in \R^n$ the following holds true
\begin{equation}  \label{MRENO}
\| \Pi_{i=1}^k \calE_i f_i \|_{L^{\frac2{k-1}}(B_R(x_0))} \lesssim_\epsilon R^{\epsilon} \Pi_{i=1}^k \| f_i \|_{L^2(U_i)},
\end{equation}
where $B_R(x_0)$ is the ball of radius $R$ centered at $x_0$. In a more compact form, this results is saying that
$\calR^*(2 \times ... \times 2 \rightarrow p, \epsilon)$ holds true. 

Closely related to this estimate, there is the multilinear Kakeya estimate. The multilinear Kakeya estimate follows from the multilinear restriction estimate by using a standard Rademacher-function argument, see \cite{BeCaTa} for details. Vice-versa, one can obtain the multilinear restriction estimate from the multilinear Kakeya estimate using a more delicate argument which incurs losses. In \cite{BeCaTa} these losses are of type $R^\epsilon$. We do not formalize the multilinear Kakeya estimate as it plays no role in our analysis; however it is important that we mention it for reference purposes. 

An alternative and shorter proof for the near-optimal multilinear Kakeya estimate was provided by Guth in \cite{Gu-easy}. An alternative and shorter proof for the multilinear restriction estimate that bypasses the use of the multlinear Kakeya estimate was provided by the author in \cite{Be1}. 

Removing the factor of $R^\epsilon$ in these estimates seems to be a very difficult problem, even in the non-endpoint case. A major breakthrough was made by 
Guth in \cite{Gu-main} where he proves the end-point case for the multilinear Kakeya estimate with no loss in $R$. In this paper Guth employs algebraic topology tools and initiates the use of the polynomial partitioning in the restriction theory that proved to be a very powerful tool, see \cite{Gu-I,Gu-II, Wa, HiRo}. 

The result of Guth in \cite{Gu-main} does not remove the factor of $R^\epsilon$ for the multilinear restriction theory, not even in the non-endpoint case. 
However, in \cite{Ben} Bennett uses Guth's result in \cite{Gu-main} to improve the loss in \eqref{MRENO} from $R^\epsilon$ to $(log{R})^\kappa$, for some large $\kappa$.

Bourgain and Guth provide in \cite{BoGu} an $\epsilon$-removal type argument that establishes \eqref{Rsp} for $p > \frac2{k-1}$, with no loss in $R$,
in the case of hypersurfaces with non-degenerate curvature. Their proof is a modification of the $\epsilon$-removal argument of Tao in \cite{Tao-BR} in a linear context. In a very recent paper, see \cite{Tao-new}, Tao removes the factor $R^\epsilon$ in the non-endpoint case for the multilinear restriction estimate in a very general setup for the hypersurfaces involved - a certain amount of regularity is required, but nothing else. This is a very involved result that further refines the heat flow method from \cite{BeCaTa}. Tao's recent result essentially leaves only the optimal problem \eqref{Rsp} with $p=\frac2{k-1}$ as an open problem. 

Our first result in this paper provides an $\epsilon$-removal argument for general hypersurfaces which have the property that the Fourier transform of the surface measure displays some decay. Precisely we assume that there exists some $\alpha >0$ such that the following decay property holds true
\begin{equation} \label{decdef}
| \calE \psi | \lesssim C(\psi) (1+|x|)^{-\alpha}
\end{equation}
for any smooth $\psi$ supported within $U$.

\begin{theo} \label{epsr}  Assume that the hypersurfaces $S_i, i=1,..,k$ are as above, satisfy the transversality condition \ref{trans} and the decay condition \ref{decdef} for some  $\alpha >0$. If $\mathcal{R}^*(2 \times ... \times 2 \rightarrow p,\epsilon)$
holds true for some $p \geq \frac2k$, then $\mathcal{R}^*(2 \times ... \times 2 \rightarrow q)$ holds true for any $q \geq p + (n+p+1)\frac{C}{\log \frac1\epsilon}$, where $C$ is any constant satisfying $C > \min(2,n-1)$. 
\end{theo}

The first thing to notice in the above statement is that our threshold for the exponent $p$ is $\frac2k$ which is lower than the conjectured one of
$\frac2{k-1}$. This improvement may seem artificial or vacuous at best, given that the conjectured optimal exponent is $\frac2{k-1}$. However $\frac2{k-1}$ is the threshold  for the generic case and it is conjectured that under appropriate curvature hypothesis on the hypersurfaces $S_i$ the threshold in \eqref{Rsp} can lowered to exponents $p < \frac2{k-1}$; in fact for some special class of hypersurfaces the author established the result up to, but excluding, the end-point $p=\frac{2(n+k)}{k(n+k-2)}$, see \cite{Be3,Be4}. 

Similar results have been obtained by Bourgain \cite{Bou-CM} and Tao and Vargas \cite{TV-CM1}, both in the context of the bilinear restriction estimate. 
When taking into account the extended range of $p$, our result is new in the context of multilinear restriction estimate where the multilinearity is trilinear or higher in order; however it is not new when $k=n$, simply because the conjectured  value of $p$ is the same in the generic case as in the non-generic case. 

Next we introduce the type of a hypersurface (an analytic condition that essentially guarantees that \eqref{decdef} holds true) and consider also the case when \eqref{decdef} holds only for some hypersurfaces, while the others are subsets of hyperplanes.  Following \cite{St} (see chapter 8, section 3.2), we define the $l$-type of a hypersurface $S=\Sigma(U)$. We fix $x_0 \in U$ and assume that there exists $k \in \N$
such that for every unit vector $\eta$ there exists $\alpha$ with $|\alpha| \leq k$ for which $|\partial^\alpha (\Sigma \cdot \eta) (x_0)| \neq 0$. The smallest such $l$ is called the type of $S$ at $x_0$. Assuming that $U$ is a subset of compact set where $\Sigma$ is defined, we define the type of $S$ to be the maximum of the types
pf the $x_0 \in U$. In the real-analytic class, the condition that $S$ has a finite type is equivalent to $S$ not being a subset of any affine hyperplane. 
Outside this class, things can get more complicated. 

As a direct consequence or by simple modifications of the argument in Theorem \ref{epsr} we obtain the following result. 

\begin{corol} \label{epsrc}
i) If all $S_i, i=1,..,k,$ are of finite type, then the result of Theorem \ref{epsr} holds true. 

ii) Assume that there exists $0 \leq k_1 < k$ such that $S_i$ are of $l_i$-type, where $l_i < \infty, \forall i=1,..,k_1$, while $S_i, i=k_1+1,..,k$ are subsets of hyperplanes. If $\mathcal{R}^*(2 \times ... \times 2 \rightarrow p,\epsilon)$
holds true for some $p \geq \frac2{k-1}$, then $\mathcal{R}^*(2 \times ... \times 2 \rightarrow q)$ holds true for any  $q \geq p + (n+p+1)\frac{C}{\log \frac1\epsilon}$, where $C$ is any constant satisfying $C > \min(2,n-1)$.

iii) If all $S_i, i=1,..,k$ are real-analytic, then the result in ii) above holds true. 

\end{corol}

The noticeable difference here is that we increased the threshold for $p$ from $\frac2k$ to $\frac2{k-1}$. This is not unnatural, given that if one of the hypersurfaces is a subset of a hyperplane, then the generic estimate in $L^{\frac2{k-1}}$ is the best that can be expected. 

When comparing the result above with the ones described earlier for the generic case, we see that our result is more general than the one of Bourgain-Guth \cite{BoGu} given that we do not impose any non-degeneracy conditions on the curvature. But it is not as strong as Tao's result in \cite{Tao-new} given that we restrict ourselves to finite type or analytic setting, while Tao's result requires only a finite amount of differentiability. We note however that the arguments in this paper are significantly simpler than the ones employed by Tao in \cite{Tao-new}; but we should also note that our arguments rely heavily on the work of Tao \cite{Tao-BR} where the $\epsilon$-removal argument is carried in the linear setup.

The second result of this paper is about the generic multilinear restriction estimate in the case when appropriate localization properties occur. 
In a nutshell, the localization properties are in small neighborhoods of submanifolds some of the $S_i$; first we need to make this rigorous. 
The localization property can be stated in terms of a parametrization, which corresponds to a localization property for $f$ used in $\calE f$, or in terms 
of the intrinsic formulation, which corresponds to a localization property for $g$ used in $\tilde \calE g$. We will state both and point out that they are in some sense equivalent.

By a submanifold $S_i'$ of $S_i$ we mean $S_i'=\Sigma_i(M_i)$ where $M_i \subset U_i \subset \R^{n-1}$ is a submanifold in $\R^{n-1}$ that is given by $\Sigma_i': U_i'  \rightarrow U_i$, where $U_i'$ is a bounded, open, connected neighborhood of the origin in $\R^{m_i}$ and $\Sigma_i'$ is smooth in the sense of \eqref{smooth}; $m_i$ is the dimension of the submanifold $S_i'$. For $0 < \epsilon \ll 1$, the $\epsilon$-neighborhood of $M_i$ is defined by $B_\epsilon(M_i)= \cup_{\xi \in M_i} B_\epsilon(\xi)$.
We say that $f$ is supported in the $\epsilon$-neighborhood of $M_i$ if it is supported in $B_\epsilon(M_i)$.

With $S_i'=\Sigma_i(M_i)$, we define the $\epsilon$-neighborhood of $S_i'$ in $\R^{n}$ in a similar manner $B_\epsilon(S_i')= \cup_{\xi \in S_i'} B_\epsilon(\xi)$. We say that $g:S_i \rightarrow \C$ is supported in the $\epsilon$-neighborhood of $S_i'$
if it is supported in $B_\epsilon(S_i') \cap S_i$. 

In what follows by $N_\zeta$ we denote the normal space to the corresponding submanifold within $\R^n$: $N_\zeta S_i'$ is the normal space to $S_i' \subset \R^n$ at $\xi \in S_i'$.

Recalling that the two formulations of the restriction operator are related via $f(\xi)= g(\Sigma(\xi)) |\frac{\partial \Sigma}{\partial_{\xi_1}} \wedge ... \wedge \frac{\partial \Sigma}{\partial_{\xi_{n-1}}}|$, and using that $\Sigma$ is a local diffeomorphism, it follows that a localization of $f$ in the $\epsilon_1$-neighborhood of $M_i$ gives a localization of $g$ in the $\epsilon_2$-neighborhood of $S_i'=\Sigma (M_i)$ and vice versa, where $\epsilon_1 \approx \epsilon_2$. 

We introduce one more notation. If $V_1,..,V_k$ are $d_k$-dimensional planes, then by $|V_1 \wedge ... \wedge V_k|$ we mean the quantity
$|v_{1,1} \wedge ... v_{1,d_1} \wedge ... \wedge v_{k,1} \wedge ... \wedge v_{k,d_k}|$ where $v_{i,1},..,v_{i,d_i}$ is an orthonormal basis in $V_i, i=1,..,k$; it is easily seen that the defined quantity is independent of the choices of orthonormal systems.

The small support condition is the following:

\begin{con} \label{C} Assume that we are given submanifolds $S_i' \subset S_i, i =1,..,k$, of codimension $c_{i}$,
with the property that there exists $\nu >0$ such that
\begin{equation} \label{normal6}
| N_{\zeta_1} S_1' \wedge  .. \wedge N_{\zeta_k} S_k' | \geq \nu
\end{equation}
for all choices $\zeta_i \in S_i'$. Given $g_i \in L^2(S_i), \forall i =1,..,k$, we assume that 
 $supp g_i \subset  B_{\mu_i} (S_i') \cap S_i$, where $0 < \mu_i \ll 1$.
\end{con}
 
We make the observation that the codimension of $S_i'$ is relative to $S_i$, thus the codimension of $S_i'$ in $\R^n$ is $c_i+1$. In particular, the normal space $N_{\zeta_i} S_i'$ has dimension $c_i+1$. As a consequence, the total number of directions in which localization is provided cannot exceed $n-k$, that is $c_1+..+c_k \leq n-k$; this follows from \eqref{normal6} and the fact that the dimension of each $N_{\zeta_i} S_i'$ is $c_i+1$. 
We also note that if $c_i=0$, then the localization property is vacuous and we simply work with $S_i'=S_i$; at the same time $N_{\zeta_i} S_i$
is just the normal to $S_i$ at $\zeta_i$. 

If the assumptions above are imposed in the generic multilinear estimate, we obtain the following result.
\begin{theo} \label{MB}
Assume $S_i, i=1,..,k$ are smooth. In addition, assume that $g_1,..,g_k$ satisfy Condition \ref{C}.
Then for any $\epsilon > 0$, there is $C(\epsilon)$ such that for every $R>0$
the following holds true
\begin{equation} \label{Lf}
\| \Pi_{i=1}^{k} \tilde \calE_i g_i \|_{L^\frac{2}{k-1}(B(0,R))} \leq C(\epsilon) \Pi_{i=1}^k \mu_i^{\frac{c_i}2} R^\epsilon \Pi_{i=1}^{k} \| g_i \|_{L^2(S_i)}.
\end{equation}
\end{theo}

A more restrictive version of the above result was established in \cite{Be1}. One of the limitations there was that only one factor had localization properties.
But the more important limitation is that the localization in \cite{Be1} was around a flat submanifold (when using a specific projection-type parametrization). 
 The more general statement in Theorem \ref{MB} will play a crucial role in a forthcoming paper of the author regarding the multilinear restriction estimate when curvature properties are taken into consideration; it is also likely that this new result will be useful in other applications.

The paper is organized as follows. In section \ref{sbr} we introduce our most common notations, recall some basic results and perform some reductions which will be useful in the latter sections. In the following two sections we prove Theorems  \ref{epsr} and \ref{MB}, respectively. 

In proving Theorem \ref{epsr} we adapt to the multilinear setup the arguments used by Tao in \cite{Tao-BR}, where an $\epsilon$-removal argument has been carried out in the linear case. In addition, the analysis involves tracking the directions in which each wave $\calE_i f_i$ propagates and the components that are essential in the application of the near-optimal multilinear restriction estimate - this is very important in the final summation process. 

In proving Theorem \ref{MB}, we use an induction on scale argument, similar to the one from \cite{Be1} used for proving a weaker version of this 
result. Dealing with the localization properties during the induction process requires a complete change in the approach (relative to \cite{Be1}) that is more robust and general.  

\section{Notation, basic results and reductions} \label{sbr}

In this section we introduce some notation and some reductions of the setup described in the introduction with the aim of simplifying the level of technicalities involved in our arguments. 

If $v_1,..,v_m$ are vectors in $\R^n$, by $span(v_1,..,v_m)$ we mean the subspace of $\R^n$ spanned by $v_1,..,v_m$. 
We use the standard notation $A \lesssim B$, meaning $A \leq C B$ for some universal $C$ which is independent of variables used in this paper; in particular
it will be independent of $\delta$ and $R$ that appear in the main proof. By $A \lesssim_N B$ we mean $A \leq C(N) B$ and indicate that $C$ depends on $N$.  

Our results involve estimates in  $L^p(S), S \subset \R^n$ with $0 < p \leq 1$.
We recall the standard estimate for superpositions of functions in $L^p$ for $0 < p \leq 1$:
\begin{equation} \label{triangle}
\| \sum_\alpha f_\alpha \|_{L^p}^p \leq \sum_\alpha \| f_\alpha \|_{L^p}^p. 
\end{equation}
 
We make the convention that the hypersurfaces involved have very small diameter in the sense that if $S=\Sigma(U)$ is a parametrization then $U \subset \R^{n-1}$ has small diameter in the classical sense.  This assumption does not affect our analysis: we can break each hypersurface in finitely many pieces 
of small diameter, run the arguments with the above setup and then sum up the result on these pieces.

Since the estimates are not affected by translating the hypersurfaces $S_i$, we can assume that they all pass through the origin,
and moreover that $0 \in S_i'$ for all $i=1,..,k$. For each $i=1,..,k$, we let $N_i=N_{0} S_i$ be the $c_i$-plane that is normal to $S_i$ at $0$. 
$N_i$ is transversal to $S_i'$ at $0$ and because $U_i'$ has small diameter, $N_i$ is transversal to $S_i'$ at every $\zeta_i \in S_i'$.

Next we proceed with reducing the problem to the case when the hypersurfaces/submanifolds involved have appropriate parameterizations; 
we do this for the more complicated setup of Theorem \ref{MB},
and note that a simplified version covers the one needed in Theorem \ref{epsr}. 

From \eqref{normal6}, we know that 
\begin{equation} \label{trans0}
|N_1 \wedge .. \wedge N_k | \geq \nu. 
\end{equation}

Our next goal is to prove that there exists a non-degenerate linear transformation $A: \R^n \rightarrow \R^n$ that allows us 
to assume that, under this transformation of the $\xi$-space, the system $N_1,..,N_k$ is an orthogonal system in the sense that if 
$v_i \in N_i, v_j \in N_j$ with $i \ne j$, then $v_i \perp v_j$.

$S_i'$ is transformed into $AS_i'$, $T_0 S_i'$ is transformed into $A T_0 S_i'$, but $AN_i$ is not necessarily the normal to $AS_i'$ at $\xi=0 \in S_i'$! Indeed, the normal condition is $\la n, v \ra=0, \forall v \in T_0 S_i, n \in N_i$ and there is no guarantee that $\la A n, A v \ra=0, \forall v \in T_0 S_i, n \in N_i$ (unless $A$ is orthonormal transformation which would not solve the problem we seek to fix).
If we wanted $e_{i,1},...,e_{i,c_i+1}$ to be the normal vectors to $AS_i$, then we need to impose $\la e_{i,j}, A v \ra=0, \forall v \in T_0 S_i$ which is equivalent to $\la A^*  e_{i,j}, v \ra=0, \forall v \in T_0 S_i$, therefore if $n_{i,j}$ was an orthonormal base for $N_i$ it suffices to impose $A^*  e_{i,j}=n_{i,j}, \forall i=1,..,k, \forall j=1,..,c_i+1$. Given an orthonormal system $\{e_{i,j}\}_{i=1,..,k; j=1,..,c_i+1}$, we can solve for $A$
satisfying these equations; from \eqref{trans0}, it follows that we can choose $A$ to obey good bounds, that is $\| A \| + \| A^{-1} \| \lesssim \nu^{-1}$.

The effect of $A$ on the other properties in Condition \ref{C} is fairly simple: $S_i'$ becomes $AS_i'$ which is a submanifold of $AS_i$,
and the effect on the neighborhood can be quantified as follows:
\[
A B_\epsilon(S_i') \subset B_{c \epsilon} (AS_i'),
\] 
for some $c$ depending on $A$.

The change of coordinates induced by $A$ will be reflected in the norms $\| g_i \|_{L^2(S_i)}$ and 
$\| \Pi_{i=1}^{k} \tilde \calE_i g_i \|_{L^\frac{2}{k-1}(B(0,R))}$ via the usual determinants of the corresponding Jacobian transformation 
- these values are both bounded from below and above in terms of $A$. Thus we can skip the use of $A$ and its related effects discussed above throughout the rest of the argument and simply work with the original objects with the additional assumption that the set of vectors $N_1,..,N_k$ is an orthogonal base in $R^n$.

For the purpose of Theorem \ref{epsr}, a similar but simpler reduction is in place, the only difference being that there $N_i$ is simply the $1$-dimensional normal space to $S_i$ at $0$, so we could work directly with (unit) vectors as opposed to spaces.

Next, we arrange a bit the variables in our problem. Each $N_i$ contains $n_i=n_i(0)$, the unit normal to $S_i$ at $0$. This particular vector plays an important role in the overall analysis, given that $S_i$ are the ambient hypersurfaces. Next in each $N_i, i=1,..,k$ we pick an orthonormal basis that contains $n_i$, that is $\{n_i\} \cup \{e^i_j\}_{j=1,..,c_i}$. If we relabel $\cup_{i=1}^k \{e^i_j\}_{j=1,..,c_i} = \{e_{k+1},..,e_l\}$, then we can further complete this system of vectors with unit vectors $e_{l+1}, ..,e_n$ with the property that the set of vectors $n_1,..,n_k,e_{k+1},..,e_{n}$ is an orthonormal system.  Then let $\xi=(\xi_1,..,\xi_n)$ where $\xi_i$ is the coordinate in the direction of $n_i$, $i=1,..,k$, and $\xi_i$ is the coordinate in the direction of $e_i$ for $i=k+1,..,n$; we also use the notation $\xi'_i=(\xi_1,.., \hat{\xi_i},..,\xi_n)$ (skip the $\xi_i$ coordinate). We let $x=(x_1,..,x_n)$ be the dual system of coordinate and $x'_i=(x_1,.., \hat{x_i},..,x_n)$.

For each $i=1,..,k$, we let $\calH_i=\{x \in \R^n: x_i=0\}$ be the hyperplane passing through the origin with normal $n_i$; we also let $\hat \calH_i=\{\xi \in \R^n: \xi_i=0\}$. We denote by $\pi_i:\R^n \rightarrow \calH_i$ the projection along $n_i$ onto $\calH_i$; similarly we denote by $\hat \pi_i:\R^n \rightarrow \hat \calH_i$ the projection along $n_i$ onto $\hat \calH_i$.

For technical purposes it is preferably to work with graph-like parametrizations of $S_i=\Sigma_i(U_i)$, that is with $\Sigma_i: U_i \subset \hat \calH_i \rightarrow \R^n$ being of the form $\Sigma_i(x)=\{ \xi \in \R^n: \xi_i=\varphi_i(\xi_i')\}, \forall \xi_i' \in U_i$ for some smooth injective map $\varphi_i : U_i \rightarrow \R$. Such parametrizations always exists under the assumption that the diameter of $U_i$ (coming from the original parametrization) is sufficiently small (which is the case): one can simply project $S_i$ onto $\calH_i \cong \R^{n-1}$ along $n_i$ to obtain $\Sigma_i^{-1} : S_i \rightarrow \hat \calH_i$. 

We let $f_i(\xi'_i)=g_i (\Sigma_i(\xi'_i)) \cdot \frac{d \sigma_{S_i}(\Sigma_i(\xi'_i))}{d \xi'_i}$; the relevance of $f_i$ comes from the fairly straightforward identity
\[
\tilde \calE_i g_i= \calE_i f_i.
\]
The above parametrization and newly constructed function $f_i$ have two additional properties. First, $\Sigma_i^{-1}(S_i')$ is a submanifold of $\R^{n-1}$ whose co-dimension equals the co-dimension of $S_i'$ relative to $S_i$.
Second, if $g_i$ has the properties in Condition \ref{C}, then $f_i$ is supported in $B_{c \epsilon} (\Sigma_i^{-1}(S_i'))$. 
In what follows we set $M_i=\Sigma_i^{-1}(S_i') \subset \hat \calH_i$ - this is the manifold whose specific neighborhood contains the support of $f_i$.

At one point in the argument we will be using more convenient versions of the $\epsilon$-neighborhood of a submanifold. As before, we can assume that $U_i'$ is of small diameter; then $S_i'$ has a graph type parametrization in the sense that, after an orthonormal change of coordinates, we have 
$\Sigma_i'(\xi')=(\xi',\varphi_{i,m_i+1} (\xi'),..,\varphi_{i,n-1}(\xi')), \xi'= (\xi_1,..,\xi_{m_i}) \in U_i'$. Thus a standard neighborhood
can be traded in this context for the simpler version $\tilde B_\epsilon (M_i)=(\xi',\varphi_{i,m_i+1} (\xi')+t_{m_i+1},..,\varphi_{i,n-1}(\xi')+t_{n-1}), \xi' \in U_i' + B_\epsilon(0),
|t_j| < \epsilon, \forall j=m_i+1,..,n-1$; here we are implicitly assuming the fact $\Sigma_i'$ extends in a neighborhood of size $\epsilon$ of $U_i'$.

The above setup reduction works the same for Theorem \ref{epsr} by simply ignoring the consideration involving the manifold $S_i'$ and its neighborhood.

With the above notations we have
\begin{equation} \label{E}
\calE_i f_i (x)= \int_{U_i} e^{i (x'_i \xi'_i + x_i \varphi_i(\xi'_i))} f_i(\xi'_i) d\xi'_i.
\end{equation}

The Fourier transform of a Schwartz function $f:\R^n \rightarrow \C$ is defined by
\[
\mathcal{F}(f)(\xi)=\hat f(\xi)= \int e^{-i x \cdot \xi} f(x) dx. 
\]
The inverse Fourier transform is defined by 
\[
\mathcal{F}^{-1}(g)(x)=\check g(x)= \frac1{(2\pi)^n} \int e^{i x \cdot \xi} g(\xi) d\xi.
\]
The standard Fourier inversion formula is in place $\mathcal{F}^{-1} \mathcal F f = f$ and $\mathcal{F} \mathcal F^{-1} g = g$. 
These definitions are then extended to distributions, in particular to $L^p(\R^n)$ spaces, in the usual manner. 

In a similar manner (just that we work in $n-1$ dimensions), we denote by $\mathcal{F}_j: \calH_j \rightarrow \hat \calH_j$ the Fourier transform, $x_j' \rightarrow \xi_j'$, and by $\mathcal{F}_j^{-1}$ the inverse Fourier transform,  $\xi_j' \rightarrow x_j'$.  Obviously, $\mathcal{F}_j, \mathcal{F}_j^{-1}$ act on the variables $x'_j,\xi'_j$ respectively.  
 
With these notations in place we make a few more observations. We denote by $\calE_j f_j(x_j',0)=\calE_j f_j|_{x_j=0}$ and note that
\begin{equation} \label{EFF}
\calE_j f_j(x_j',0) = (2\pi)^{n-1} \mathcal{F}_j^{-1} f_j. 
\end{equation}
With the parametrizations introduced earlier for each $S_i$, it follows that
\begin{equation} \label{calei}
\calE_j f_j = e^{i x_j \varphi_j(D'_j)} \check{f}_j,
\end{equation}
where the symbol of $e^{i x_j \varphi_j(D'_j)}$ is $e^{ix_j \varphi_{j}(\xi_j')}$. 

Another observation is that the operator $\calE_jf_j(\cdot,x_j): L^2(\calH_j) \rightarrow L^2(\calH_j)$ is an $L^2$ isometry with respect to $x_j$, that is
\begin{equation} \label{mc}
\| \calE_j f_j(\cdot,x_j) \|_{L^2}= \| \calE_j f_j(\cdot,0) \|_{L^2}=(2\pi)^{\frac{n-1}2} \|f_j\|_{L^2}, \quad \forall x_j \in \R.
\end{equation}
   
For the operators $\calE_j$ we highlight a commutator estimate which will be used in our proof.   We define the operator $\nabla \varphi_j(\frac{D'}i)$ to be multiplier with symbol $\nabla \varphi_j(\xi')$. 
 For any $x \in \R^{n}$ and $c \in \calH_j$, it holds true that
\begin{equation} \label{com}
(x'_j-c-x_{j} \nabla \varphi_j(\frac{D'}i))^N \calE_j f (x)=  \calE_j (\mathcal{F}_j( (x'_j-c)^N  \mathcal{F}_j^{-1} f)) (x), \quad \forall N \in \N. 
\end{equation}

Next we prepare some geometric elements that are needed in the proof. We let $\mathcal{L}:=\{ z_1 n_1 + ... +z_{k} n_{k} + z_{k+1} e_{k+1} + ..+z_n e_n:(z_1,..,z_{n}) \in \Z^{n} \}$ be the standard lattice in $\R^{n}$ generated by the vectors $n_1,..,n_{k},e_{k+1},..,e_n$. In each $\calH_i, i=1,..,k$ we construct the induced lattice $\mathcal{L}(\calH_i)=\pi_{i} (\mathcal{L})$. Recall that for each $i=1,..,k$, it holds true that 
$\{ n_1,..,n_k,e_{k+1},..,e_l\} \setminus \{n_i\} \subset \calH_i$. We split $\calH_i = \calH_i'  \oplus \calH_i''$, where $\calH_i''$ is spanned by the set of vectors $\{ e^i_1,..,e^i_{c_i} \}$, while $\calH_i'$ is spanned by $\{ n_1,..,n_k,e_{k+1},..,e_l\} \setminus \{N_i, e^i_1,..,e^i_{c_i} \}$. Correspondingly,
this produces a split of $\mathcal{L}(\calH_i)$ as follows  $\mathcal{L}(\calH_i)=\mathcal{L}(\calH_i') \oplus \mathcal{L}(\calH_i'')$.

Given $r > 0$ we define $\calC(r)$ be the set of of cubes of size $r$ in $\R^{n}$ that are centered at points in the lattice $r \mathcal{L}$; a cube in $\calC(r)$ has the following form $q(\bj): = [r (j_1-\frac12), r(j_1+\frac12)] \times .. \times [r (j_{n}-\frac12), r(j_{n}+\frac12)]$
where $\bj=(j_1,..,j_{n}) \in \Z^{n}$. For such a cube we define $c(q)=r \bj=(r j_1,.., r j_{n}) \in r\mathcal{L}$ to be its center; 
vice-versa, if $c \in r\mathcal{L}$ then we define $q(c)$ to be the cube whose center $c(q)=c$. 
Then, for each $i=1,..,n$, we let $\calC \calH_i(r) =\pi_{i} \calC(r)$ be the set of cubes of size $r$ in the hyperplane $\calH_i$.
Finally, given two cubes $q,q' \in \calC(r)$ (or $\calC \calH_i(r)$ or $\calC \calH'_i(r), \calC \calH''_i(r)$)  we define
$d(q,q')$ to be the distance between them when considered as subsets of the underlying space, be it $\R^{n}$ or $\calH_i$ or $\calH_i'/\calH_i''$.

Let $\chi_0^n:\R^n \rightarrow [0,+\infty)$ be a Schwartz function, normalized in $L^1$, that is $\| \chi_0^n \|_{L^1}=1$,
and with Fourier transform supported on the unit ball. For each $q \in \calC(r)$, define
$\chi_q: \R^n \rightarrow \R$ by
\[
\chi_{q}(x) = \chi_0^n (\frac{x-c(q)}r)
\]
Notice that $\hat \chi_{q}$ is supported in $B(0,r^{-1})$. By the Poisson summation formula and the properties of $\chi_0^n$, 
\begin{equation} \label{pois}
\sum_{q \in \calC(r)} \chi_{q}=1.
\end{equation}
Using the properties of $\chi_q$, a direct exercise shows that for each $N \in \N$, the following holds true
\begin{equation} \label{SN}
\sum_{q \in \calC(r)}  \| \la \frac{x-c(q)}r \ra^{N} \chi_{q} g \|_{L^2}^2 \lesssim_N \| g \|_{L^2}^2
\end{equation}
for any $g \in L^2(\R^n)$.

We introduce similar entities in each $\calH_i, i=1,..k$.  For each $q \in \calC \calH_i(r)$, define
$\chi_q: \calH_i \rightarrow \R$ by $\chi_{q}(x) = \chi_0^{n-1} (\frac{x-c(q)}r)$ and note that $\mathcal{F}_i \chi_{q}$ is supported in $B(0,r^{-1})$. 
By the Poisson summation formula and the properties of $\chi_0^{n-1}$, 
\begin{equation} \label{poisl}
\sum_{q \in \calC \calH_i(r)} \chi_{q}=1.
\end{equation}
In a similar way to \eqref{SN}, the following holds true
\begin{equation} \label{SNl}
\sum_{q \in \calC \calH_i(r)}  \| \la \frac{x-c(q)}r \ra^{N} \chi_{q} g \|_{L^2}^2 \lesssim_N \| g \|_{L^2}^2
\end{equation}
for any $g \in L^2(\calH_i)$. Here, the variable $x$ is the argument of $g$ and belongs to $\calH_i$.

We recall from \cite{Be1} the following discrete version of the continuous Loomis-Whitney inequality:
\begin{equation} \label{LWd}
\| \Pi_{i=1}^{k} g_i(\pi_{i}(z)) \|_{l^{\frac2{k-1}}(\mathcal{L})} \ls \Pi_{i=1}^{k} \| g_i \|_{l^2(\mathcal{L}(\calH_i))}. 
\end{equation}
We need a further refinement of this estimate.

For a function $g : \mathcal{L}(\calH_i) \rightarrow \C$, we define the space $l^2l^\infty(\mathcal{L}(\calH_i') \times \mathcal{L} (\calH_i''))$ to be the space of functions 
whose norm
\[
\| g \|_{l^2l^\infty(\mathcal{L}(\calH_i') \times \mathcal{L} (\calH_i''))} 
= \left( \sum_{z' \in \mathcal{L}(\calH_i')} \sup_{z'' \in \mathcal{L}(\calH_i'')} | g(z',z'')|^2 \right)^\frac12, 
\]
is finite.  With this notation in place we have the following result:

\begin{lema} \label{Le3}
Assume $g_i \in l^2l^\infty(\mathcal{L}(\calH_i') \times \mathcal{L} (\calH_i'')), i=1,..,k$. Then the following holds true
\begin{equation} \label{LWd3}
\| \Pi_{i=1}^{k} g_i(\pi_{i}(z)) \|_{l^{\frac2{k-1}}(\mathcal{L})} \ls 
 \Pi_{i=1}^{k} \| g_i \|_{l^2l^\infty(\mathcal{L}(\calH_i') \times \mathcal{L} (\calH_i''))}. 
\end{equation}
\end{lema}

\begin{proof} For $z \in \mathcal{L}$ we write $z=(z',z'',z''')$ where $z'=(z_1,..,z_{k})$ collects the coordinates in the directions of $n_1,..,n_k$,
 $z''$ collects the coordinates in the directions of $e_{k+1},..,e_{l}$ and $z'''$ collects the coordinates in the directions of $e_{l+1},..,e_{n}$. 
 In this proof, we need to further refine the latices $\calH_i$. We let $\tilde \calH_i''$ be the projection of $\calH_i$ onto the subspace generated by the vectors $\{e_{k+1},..,e_l\} \setminus \{e^i_1,..,e^i_{c_i}\}$. 
 
 We fix $z'$ and $z'''$, let $ \{z' \} \times \mathcal{L}'' \times \{z'''\}$ be the  sub-lattice of $\mathcal{L}$ obtained by fixing $z'$ and $z'''$.
  Then we use H\"older's inequality to obtain
\[
\|  \Pi_{i=1}^{k} g_i(\pi_{i}( z', \cdot, z''' )) \|_{l^{\frac2{k-1}}(\{z' \} \times \mathcal{L}'' \times \{z'''\})} \ls 
\Pi_{i=1}^{k} \| g_i(\pi_{i}(z',\cdot,z''')) \|_{l^2l^\infty(\mathcal{L}(\tilde \calH_i'') \times \mathcal{L}(\calH_i''))}.
\]
In justifying this estimate, we have also used the fact that $l^2l^\infty(\mathcal{L}(\tilde \calH_i'') \times \mathcal{L}(\calH_i''))$ is the strongest 
norm that combines an $l^2$ norm in the variables from $\mathcal{L}(\tilde \calH_i'') $ and an $l^\infty$ norm in the variables from $ \mathcal{L}(\calH_i''))$; 
for instance it controls $l^\infty l^2( \mathcal{L}(\calH_i'') \times \mathcal{L}( \tilde \calH_i''))$. 

Then we apply \eqref{LWd} with respect to the variable $z'$ to obtain
\[
\|  \Pi_{i=1}^{k} g_i(\pi_{i}( \cdot, \cdot, z''' )) \|_{l^{\frac2{k-1}}(\mathcal{L}' \times \mathcal{L}'' \times \{z'''\})} \ls 
\Pi_{i=1}^{k} \| g_i(\pi_{i}(z',\cdot,z''')) \|_{l^2l^\infty(\mathcal{L}(\tilde \calH_i'') \times \mathcal{L}(\calH_i''))}.
\]
Applying H\"older with respect to the $z'''$ variable gives
\[
\|  \Pi_{i=1}^{k} g_i \circ \pi_{i} \|_{l^{\frac2{k}}_{z'''} l^{\frac2{k-1}}_{z',z''}} \ls 
\Pi_{i=1}^{k} \| g_i \circ \pi_{i} \|_{l^2l^\infty(\mathcal{L}(\calH_i') \times \mathcal{L} (\calH_i''))}.
\]
To conclude with \eqref{LWd3} we only need to use the simple inequality $\| a \|_{l^{\frac2{k-1}}} \ls \| a \|_{l^{\frac2{k}}}$ for any sequence $a \in \Z^m$
and in any dimension $m$.

 \end{proof}

\section{The $\epsilon$-removal result} \label{epsrl}
In this Section we provide a proof of Theorem \ref{epsr}, a theorem  that provides an $\epsilon$-removal technique for multilinear estimates. The closest type of result in this direction is the work of Tao and Vargas in \cite{TV-CM1} in the context of bilinear restriction estimates. Unfortunately we cannot adapt that approach to our context, since, among other issues, it relies on duality type arguments; in our particular context we work with $L^p$ quasi-norms with $p<1$ and use of duality is not an option. Instead we follow a previous $\epsilon$-removal argument of Tao in \cite{Tao-BR} in a linear context and use techniques developed by the author in see \cite{Be1}
to adapt it in the context of multilinear estimates. 

Our first result is of a technical nature and its use will become clear once we proceed with the proof of Theorem \ref{epsr}. The basic idea is that we need to use the near-optimal result 
\[
\| \Pi_{i=1}^k \calE_i f_i \|_{L^p(B_R)} \lesssim R^\epsilon \Pi_{i=1}^k \| f_i \|_{L^2}
\]
in a more optimal form that essentially collects the mass of each wave $\calE_i f_i$ from $B_R(x_i)$. The starting point is that the mass of $f_i$
can be recovered in multiple ways based on \eqref{mc}; as a consequence we can rewrite the above as
\[
\| \Pi_{i=1}^k \calE_i f_i \|_{L^p(B_R)} \lesssim R^\epsilon \Pi_{i=1}^k \| \calE_i f_i|_{x_i=c_{Q,i}} \|_{L^2(\calH_i)}. 
\]
On the other hand, information coming from cubes $q \times \{x_i=c_{Q,i}\}, q \in \calC \calH_i(R)$ that are at distance $\gg R$ from $\pi_i Q \times \{x_i=c_{Q,i}\}, q \in \calC \calH_i(R)$ should not impact much the estimate, due to the finite speed of propagation. The following result encodes this heuristics. 

\begin{lema}
Let $Q \in \calC(R)$. Then the following holds true
\begin{equation} \label{lEF}
\| \Pi_{i=1}^k \calE_i f_i \|_{L^p(Q)} \lesssim R^\epsilon \Pi_{i=1}^k \left( \sum_{q \in \calC \calH_i (R)} \la \frac{d(\pi_{i} Q, q)}R \ra^{-(2N-n^2)} \| \la \frac{x_i'-c(q)}R \ra^{N} \chi_{q} \calE_i f_i|_{x_i=c_{Q,i}} \|^2_{L^2(q)} \right)^\frac12. 
\end{equation}
\end{lema}

There is one technicality that we skip in the above statement in order to keep it simple. The occurrence of $\chi_q$ in various places has the effect
of modifying the support of entities involving $f_i$ by a factor of $R^{-1}$. This would require us to assume that the original near-optimal result holds for each $f_i$ supported in $U_i + B(0,R^{-1})$ in order to claim the above for each $f_i$ supported in $U_i$. The practical effect in the final argument here is that in order to obtain the $\epsilon$-removal result under the hypothesis that each $f_i$ is supported in $U_i$, we need to know the near-optimal result for each $f_i$ supported in $U_i + B(0,c)$ for some small $0 < c \ll 1$. 

\begin{proof} We just sketch the main steps in this argument, as the argument is very similar to the one used for proving the more difficult $\eqref{INS}$; in the present context the $\delta$ that appears in $\eqref{INS}$ should is essentially $\approx 1$.

The first observation is that, without restricting the generality of the argument we can assume that 
$Q$ is centered at the origin, that is $c_{Q}=0$.  The above then becomes
\begin{equation} \label{lEFr}
\| \Pi_{i=1}^k \calE_i f_i \|_{L^p(Q)} \lesssim R^\epsilon \Pi_{i=1}^k \left( \sum_{q \in \calC \calH_i (R)} \la \frac{c(q)}R \ra^{-(2N-n^2)} \| \la \frac{x_i'-c(q)}R \ra^{N} \chi_{q} \calE_i f_i|_{x_i=0} \|^2_{L^2(q)} \right)^\frac12. 
\end{equation}
 
We establish the improvement one term at the time, starting with $\calE_1f_1$. Since $\varphi_1(0)=0$ and $\nabla \varphi_1(0)=0$, we work with the basic multiplier $x_1'$ and \eqref{com} to justify the following identity
\[
\begin{split}
x_1' \calE_1 \mathcal{F}_1 (\chi_q  \mathcal{F}_1^{-1} f_1)& = (x_1'- x_{1} \nabla \varphi_1(\frac{D'}i))  \calE_1 \mathcal{F}_1 (\chi_q  \mathcal{F}_1^{-1} f_1) + x_{1} \nabla \varphi_1(\frac{D'}i)  \calE_1 \mathcal{F}_1 (\chi_q  \mathcal{F}_1^{-1} f_1) \\
& = \calE_1 (\mathcal{F}_1( x'_1  \chi_q \mathcal{F}_1^{-1} f_1)) (x) + x_{1} \nabla \varphi_1(\frac{D'}i)  \calE_1 (\mathcal{F}_1( \chi_q \mathcal{F}_1^{-1} f_1)).
\end{split}
\]
Arguing in a similar manner to the way we do in the prof of \eqref{INS} we derive \eqref{lEFr}; the details are left as an exercise. One important observation is that \eqref{EFF} provides the relation between $\calE_1 f_1|_{x_1=0}$ and $\mathcal{F}_1 f_1$ that allows us to switch from using one of them to the other one. 

\end{proof}

We now proceed with the proof of Theorem \ref{epsr}.  Without restricting the generality of the argument, we can assume that $\|f_i\|_{L^2}=1, i=1,..,k$. We do this by  providing a weak type estimate, that is if
\[
E(\lambda)= \{ x \in \R^n: |\Pi_{i=1}^k \calE_i f_i| \geq \lambda \},
\]
it suffices to prove $|E(\lambda)| \lesssim \lambda^{-q}$ for some $q$ satisfying $ q < p + (n+p+1)\frac{D}{\log \frac1\epsilon}$.

The first observation is that if $x \in E(\lambda)$, then for some universal $c \approx 1$, $B_{c \lambda} (x) \subset E(\frac{\lambda}2)$. This follows from the simple computation
\[
 |\Pi_{i=1}^k \calE_i f_i(x) -  \Pi_{i=1}^k \calE_i f_i(x_0) | \leq k C_1^{k-1} \max_{i} |\calE_i f_i(x)- \calE_i f_i(x_0)| \leq k C_1^{k-1} C_2 |x-x_0|
\]
where
\[
C_1= \max_i \| \calE_i f_i \|_{L^\infty} \lesssim 1, \quad C_2 = \max_i \| \nabla \calE_i f_i \|_{L^\infty} \lesssim 1.
\]
Thus if we choose $c= \min (1,\frac12 (k C_1^{k-1} C_2)^{-1}) \approx 1$, it follows that $|\Pi_{i=1}^k \calE_i f_i(x) -  \Pi_{i=1}^k \calE_i f_i(x_0) | \leq \frac{\lambda}2$ and the conclusion follows. 

As a consequence, if we let $F=\cup_{x \in E(\lambda)} B_{c \lambda}(x)$ we obtain that $F \subset E(\frac{\lambda}2)$; we let $G=\cup_{x \in E(\lambda)} B_{1}(x)$ and note that $|G| \lesssim  \lambda^{-n} |F|$. 

Next we use the following result due to Tao \cite{Tao-BR}. We say that a collection $\{B_R(x_i)\}_{i=1}^N$ is sparse if the centers $x_i$ are $R^C N^C$ separated. 

\begin{lema} [Lemma 3.3, \cite{Tao-BR}]Suppose $E$ is the union of $c$-cubes and $N \geq 1$. Then there exists $O(N|F|^\frac1N)$ sparse collections of balls that cover $G$, such that the balls in each collection have radius $1 \lesssim R \lesssim |E|^{C^N}$. 
\end{lema}

It is easy to see that the above statement can be made in terms of cubes, rather than balls, and this is how we use it below. Thus the above lemma gives us 
$O(N|G|^\frac1N)$ sparse collections $O_l$ of cubes that cover the set $G$ (and hence the set $F$), such that the cubes in each collection $O_l$  have radius/side-length $1 \lesssim R_l \lesssim |E|^{C^N}$. We fix a sparse collection of cubes $O_l$; for each cube $Q \in O_l$ we apply the near-optimal result in its refined version \eqref{lEF}:
\[
\| \Pi_{i=1}^k \calE_i f_i \|_{L^p(Q)} \lesssim R_l^\epsilon \Pi_{i=1}^k \left( \sum_{q \in \calC \calH_i (R)} \la \frac{d(\pi_{i} Q, q)}R \ra^{-(2N-n^2)} \| \la \frac{x_i'-c(q)}R \ra^{N} \chi_{q} \calE_i f_i|_{x_i=c_{Q,i}} \|^2_{L^2(q)} \right)^\frac12.  
\]
Based on Lemma \ref{sparse} part ii) below, we can conclude that each of the terms on the right hand-side has the $l^2_{Q \in O_l}$ norm bounded by $\|f_i\|_{L^2}$; thus their product has $l^{\frac{2}{k}}_{Q \in O_l}$ norm bounded by $\Pi_{i=1}^k \|f_i\|_{L^2}$ and since $p \geq \frac2{k}$ it follows that
\[
\| \Pi_{i=1}^k \calE_i f_i \|_{L^p(O_l)} \lesssim R_l^\epsilon \Pi_{i=1}^k \|f_i\|_{L^2}=R_l^\epsilon. 
\]
Since $F \subset E(\frac{\lambda}2)$, it follows that
\[
\begin{split}
| F | \lesssim & \lambda^{-p} \| \Pi_{i=1}^k \calE_i f_i \|^p_{L^p(F)} \lesssim \lambda^{-p} \sum_l \| \Pi_{i=1}^k \calE_i f_i \|_{L^p(O_l)}^p \\
\lesssim  & \lambda^{-p}  \sum_l R_l^{p \epsilon} \lesssim \lambda^{-p} N|G|^\frac1N (|G|^{C^N})^{p \epsilon} \\
= & \lambda^{-p} N |G|^{\frac1N+ p \epsilon C^N} \lesssim \lambda^{-p} N ( \lambda^{-n} |F|)^{\frac1N+ p \epsilon C^N}. 
\end{split}
\] 
We let $\beta=\frac1N+ p \epsilon C^N$ and derive the following estimate
\[
|F| \lesssim N \lambda^{-\frac{p+n \beta}{1-\beta}}. 
\]
For $\epsilon \ll e^{-C} \leq 1$, we make the choice $N=C^{-1} \log{\frac1\epsilon}$; then it follows  $\frac{C}{\log{\frac1\epsilon}} \leq \beta \leq \frac{2C}{\log{\frac1\epsilon}} \ll 1$ and the following holds true $\frac{p+n\beta}{1-\beta} < p+(n+p+1) \beta \leq p+ (n+p+1) \frac{2C}{\log{\frac1\epsilon}}$. Thus our argument 
for Theorem \ref{epsr} is complete provided that we establish the following result.

\begin{lema} \label{sparse} 
i) Assume that $S=\Sigma(U)$ is a smooth hypersurface that is parametrized by $\xi_n= \varphi(\xi_n')$. In addition, we assume that there exists 
$\alpha > 0$ such that
\begin{equation} \label{decay1}
| \calE \psi | \lesssim C(\psi) (1+|x|)^{-\alpha}
\end{equation}
for any smooth $\psi$ supported within $U$. 

Given a $C$-sparse family of cubes $\{ Q_j \}_{j=1}^N$, for some $C > \frac{\min(2,n-1)}\alpha$, the following holds true
\begin{equation} \label{spre}
\sum_{j=1}^N \| \calE f|_{x_n=c_n(Q_j)} \|_{L^2(\pi_n(Q_j))}^2 \lesssim \| f \|_{L^2}^2,
\end{equation}
for any $f$ supported within $U$.

ii) Under similar hypothesis as in i), the following holds true:  
\begin{equation} \label{decay2}
\sum_{j=1}^N \sum_{q \in \calC \calH_i (R)} \la \frac{d(\pi_{i} Q_j, q)}R \ra^{-(2N-n^2)} \| \la \frac{x_i'-c(q)}R \ra^{N} \chi_{q} \calE_i f_i|_{x_i=c_{Q,i}} \|^2_{L^2(q)} \lesssim \|f_i\|_{L^2}^2. 
\end{equation}
\end{lema}

\begin{proof} The estimate above can be restated as follows. If $g= \check f$, then from \eqref{calei}, it follows that \eqref{spre} is equivalent to
\begin{equation} \label{spre1}
\sum_{j=1}^N \| \chi_{\pi_n Q_j} e^{i c_n(Q_j) \varphi(D)} g \|_{L^2(\pi_n(Q_j))}^2 \lesssim \| g \|_{L^2}^2,
\end{equation}
for any $g$ that is supported within $U$. We let $T: L^2(\R^n) \rightarrow l^2 L^2(\R^n)$ be defined by
\[
Tf=(\chi_{\pi_n Q_j} e^{i c_n(Q_j) \varphi(D)} g )_{j=1}^N.
\]
Then \eqref{spre} is equivalent to proving that $\| T \| \lesssim 1$, where by the operator norm we mean the norm of $T$ acting from $L^2(\R^n) \rightarrow l^2 L^2(\R^n)$. Its adjoint operator $T^*: l^2 L^2(\R^n)  \rightarrow L^2(\R^n)$ is given by
\[
T^*F=\sum_{j=1}^N e^{- i c_n(Q_j)  \varphi(D)} ( \chi_{\pi_n Q_j}  F_j).
\]
Thus, it suffices to establish that the operator $T T^*: l^2 L^2(\R^n)  \rightarrow l^2 L^2(\R^n)$ has a good bound on its norm, that is $\| T T^* \| \lesssim 1$
$TT^*$ is given by
\[
T T^*F=(\chi_{\pi_n Q_k} e^{- i c_n(Q_k) \varphi(D)} \sum_{j=1}^N e^{ i c_n(Q_j) \varphi(D)} \chi_{\pi_n Q_j} F_j )_{k=1}^N.
\]
In the above we make the additional observation that each $F_j$ has its Fourier transform $\hat F_j$ is supported in a small neighborhood (of size $O(R^{-1})$) 
of $U$;  this is a consequence of the original localization property of the function $g$, namely that $\hat g$ was supported in $U$. Formally we could have rigorously 
kept track of this by inserting an additional operator $\chi_U(D)$ (where $\chi_U$ is the characteristic function of the set $U$  in the definition of $T$, but this would have complicated the exposition without bringing any useful insight into the argument. 

We claim the following estimate:
\begin{equation}\label{duales}
\| \chi_{\pi_n Q_k} e^{- i (c_n(Q_k)- c_n(Q_j))  \varphi(D)}  \chi_{\pi_n Q_j} F_j \|_{L^2} \lesssim min(1, R^{n-1} \la c(Q_k) - c(Q_j) \ra^{-\alpha}) \| F_j \|_{L^2}.
\end{equation}
Using the sparseness of the set of cubes with $C > \frac{\min(2,n-1)}{\alpha}$, it follows that $TT^*$ is bounded and we can conclude the proof of \eqref{spre}.

We now prove \eqref{duales}. Using the full translation invariance of the estimate, \eqref{duales} is equivalent to 
\begin{equation}\label{duales1}
\| \chi_{ Q(c_n')} e^{- i c_n  \varphi(D)}  \chi_{Q(0)} F_j \|_{L^2} \lesssim min(1,R^{n-1} \la c \ra^{-\alpha}) \| F_j \|_{L^2}.
\end{equation}
To get the bound with constant $1$, we simply use the fact that all operators above are bounded on $L^2$. To get the second bound we start with the explicit formula
\[
e^{- i c_n  \varphi(D)}  \chi_{Q(0)} F_j = K(\cdot, c_n) \ast (\chi_{Q(0)} F_j), 
\]
with
\[
K(x_n',x_n)= \int e^{i (x_n' \xi_n'+ x_n \varphi(\xi_n'))} \chi(\xi_n') d\xi_n'
\]
where $\chi(\xi_n')$ keeps track of the localization property of each $F_j$. From \eqref{decay1} it follows that $K$ satisfies the decay properties:
\[
|K(x)| \lesssim (1+|x|)^{-\alpha}. 
\]
Therefore
\[
\| K \|_{L^1(Q_R(c_n'))} \lesssim R^{n-1} (1+\|c\|)^{-\alpha}. 
\]
From this we obtain the \eqref{duales1} with constant $R^{n-1} (1+\|c\|)^{-\alpha}$. 

\end{proof}

We end this section with the proof of Corollary \ref{epsrc}. From \cite{St} (see chapter 8, section 3.2, Theorem 2) it follows that if $S$ has $l$-type,
then $|\calE \psi(x) | \lesssim (1+|x|)^{-\frac1l}$. Then i) follows from Theorem \ref{epsr}. 

For ii), the basic idea is that for the hypersurfaces $S_1,..,S_{k_1}$ we can still use parts of the analysis above, while for the $S_{k_1+1},..,S_{k}$ planar hypersurfaces we will make use of the known trivial estimates. 

The argument follows the same steps as the one used in the proof of Theorem \ref{epsr}. When we arrive at the use of the near-optimal result in its refined version \eqref{lEF}:
\[
\| \Pi_{i=1}^k \calE_i f_i \|_{L^p(Q)} \lesssim R_l^\epsilon \Pi_{i=1}^k \left( \sum_{q \in \calC \calH_i (R)} \la \frac{d(\pi_{i} Q, q)}R \ra^{-(2N-n^2)} \| \la \frac{x_i'-c(q)}R \ra^{N} \chi_{q} \calE_i f_i|_{x_i=c_{Q,i}} \|^2_{L^2(q)} \right)^\frac12,
\]
we cannot make use of Lemma \ref{sparse} for all hypersurfaces involved, thus we cannot recover the estimate in $l^{\frac{2}{k}}_{Q \in O_l}$ bounded by $\Pi_{i=1}^k \|f_i\|_{L^2}$. What we can do instead, is to recover an estimate in $l^{\frac{2}{k_1}}_{Q \in O_l}$ for the product of the first $k_1$
terms on the right-hand side with a bound of $\lesssim \Pi_{i=1}^{k_1} \|f_i\|_{L^2}$. For the other terms we rely on an exact calculus. As we described in the section \ref{sbr}, we have chosen the coordinates such that $n_{k_1+1},..,n_{k}$, the normals at the hyperplanes where $S_{k_1+1},..,S_{k}$ respectively lie, belong to the set of coordinate directions. As a consequence, $\calE_i f_i|_{x_i=c_{Q,i}}=\calE_i f_i|_{x_i=0}$. Thus the remaining term to be estimated is 
\[
\Pi_{i=k_1+1}^k \left( \sum_{q \in \calC \calH_i (R)} \la \frac{d(\pi_{i} Q, q)}R \ra^{-(2N-n^2)} \| \la \frac{x_i'-c(q)}R \ra^{N} \chi_{q} \mathcal{F}_i^{-1} f_i|_{x_i=0} \|^2_{L^2(q)} \right)^\frac12.
\]
It is an easy exercise to see that we obtain an $l^{\frac2{k-k_1-1}}$ estimate for the above term with a bound of 
$\lesssim \Pi_{k_1+1}^k \|f_i\|_{L^2}$ - morally this is at the level of Loomis-Whitney inequality. Alternatively, 
each term above has $l^2l^\infty (R \mathcal{L} (\calH_i) \times R span(N_i))$ with respect to $Q \in \calC(R)$ and with a bound of $\lesssim \|f_i\|_{L^2}$. This leads to the estimate $l^{\frac2{k-k_1-1}}$ as above using similar but simpler arguments to the ones in Lemma \ref{Le3}. 

Finally iii) follows from ii). 

\section{Proof of Theorem \ref{MB}} \label{MBP}

The proof uses an induction of scales argument. When moving from smaller scales to the larger ones, it is very useful
to have tight control of the leaks of mass between various spatial regions for each individual wave $\calE_i f_i$. 
This is achieved by splitting each domain $U_i$ into smaller pieces of diameter $\leq \delta$, for some $0 < \delta \ll 1$. 
This, in turn, splits the surfaces $S_i=\Sigma_i(U_i)$ in the corresponding pieces. 
 It suffices to prove the multilinear estimate for each $S_i$ being replaced
by one of its pieces, since then we can sum up the estimates for all possibles combinations of pieces using \eqref{triangle} and generate the original estimate at a cost of picking a factor that is a power of $\delta^{-1}$. In the end of the argument, $\delta$ will be chosen in terms of absolute constants and $\epsilon$, but not $R$, and the power of $\delta^{-1}$  will be absorbed into $C(\epsilon)$. This idea originates from the work of Guth in \cite{Gu-easy} and the author later used it in \cite{Be1,Be3,Be4}. 

Once we have decided to work with these pieces from each $S_i$, we run the argument in section \ref{sbr}: we translates the pieces so that $0 \in S_i$, and redefine all the entities there.

 The proof  of Theorem \ref{MB} relies on estimating $\Pi_{i=1}^{k} \calE_i f_i$ on cubes on the physical side and analyze how the estimate
behaves as the size of the cube goes to infinity by using an inductive type argument with respect to the size of the cube.
As we move from one spatial scale to a larger one, we will have to tolerate slightly larger Fourier support in the argument. But this accumulation is in the form of a convergent geometric series, therefore the only harm it does is imposing an additional technical layer in the argument. 

\begin{defin} \label{DefI}
Given $R \geq \delta^{-2}$ we define $A(R)$ to be the best constant for which the estimate
\begin{equation}
\| \Pi_{i=1}^{k} \calE_i f_i \|_{L^{\frac2{k-1}}(Q)} \leq A(R,\delta,\mu) \Pi_{i=1}^{k} \| f_i \|_{L^2} 
\end{equation}
holds true for all cubes $Q \in \calC(R)$, under the assumption that $f_i$ is supported in the set 
\[
B_{\mu_i+10R^{-1}}(M_i) \cap B(0,\delta), \quad \forall i=1,..k.
\]
\end{defin}

In the above we have used the notation $\mu=(\mu_{1},..,\mu_k)$. The induction starts from $R \geq \delta^{-2}$ in order to be able to propagate the support hypothesis.  We also tacitly assume that $\mu_i \ll \delta^{-2}, \forall i=1,..,k$, or else the final gain of $\mu_i^{\frac{c_i}2}$ is indistinguishable from powers of $\delta^{-1}$ which will naturally contribute to the final constant $C(\epsilon)$ that appears in the min estimate \eqref{Lf}. 

We provide an estimate inside any cube $Q \in \calC(\delta^{-1}R)$ based on prior information
on estimates inside cubes $q \in \calC(R) \cap Q$. Without restricting the generality of the argument, we assume
that $Q$ is centered at the origin and recall that each $q \in \calC(R) \cap Q$ has its center in
$R \mathcal{L}$. When such a  $q$ is projected using $\pi_{i}$ onto $\calH_i$ one obtains $\pi_{i} q \in \calC \calH_i(R)$.

Each $q \in \calC(R) \cap Q$ has size $R$ and the induction hypothesis is the following:
\begin{equation} \label{IO}
\| \Pi_{i=1}^{k} \calE_i f_i \|_{L^{\frac2{k-1}}(q)} \leq A(R,\delta,\mu) \Pi_{i=1}^{k} \| f_i \|_{L^2}. 
\end{equation}
We strengthen this as follows:
\begin{equation} \label{INS}
\begin{split}
& \| \Pi_{i=1}^{k} \calE_i f_i \|_{L^{\frac2{k-1}}(q)} \les  A(R,\delta,\mu) \\
 & \cdot \Pi_{i=1}^{k} \left( \sum_{q' \in \calC \calH_i(R)} \la \frac{d(\pi_{i} q,q')}R \ra^{-(2N-n^2)} \| \la \frac{x_i'-c(q')}R \ra^{N} \chi_{q'} \mathcal{F}_i^{-1} f_i \|_{L^2}^2 \right)^\frac12. 
\end{split}
\end{equation}

The basic idea in \eqref{INS} is the following: if $q' \ne \pi_{i} q$, then 
$\calE_1 \mathcal{F}_1 ( \chi_{q'}  \mathcal{F}_1^{-1} f_1)$ has off-diagonal type contribution outside $q' \times [-\delta^{-1} R, \delta^{-1} R]$ (the interval stands for the $i$'th slot), thus it has off-diagonal type contribution to the left-hand side of \eqref{INS}. We now turn to the details and fix $i=1$ and $q' \in \calC \calH_1(R)$. With $x=(x_1,x'_1)$ we have
\[
\begin{split}
& \| (x_1'- c(q')-x_1 \nabla \varphi_1(\xi'_0)) \calE_1 \mathcal{F}_1 ( \chi_{q'}  \mathcal{F}_1^{-1} f_1) \cdot \Pi_{i=2}^{k}  \calE_i f_i \|_{L^{\frac2{k-1}}(q)} \\
\leq & \| (x_1'- c(q')-x_1 \nabla \varphi_1(\frac{D'}i)) \calE_1 \mathcal{F}_1 ( \chi_{q'} \mathcal{F}_1^{-1} f_1) \cdot \Pi_{i=2}^{k}  \calE_i f_i  \|_{L^{\frac2{k-1}}(q)} \\
+ & \| x_1 ( \nabla \varphi_1(\xi'_0) - \nabla \varphi_1(\frac{D'}i)) \calE_1 \mathcal{F}_1 ( \chi_{q'} \mathcal{F}_1^{-1} f_1) \cdot \Pi_{i=2}^{k}  \calE_i f_i  \|_{L^{\frac2{k-1}}(q)} \\
= & \|  \calE_1 \mathcal{F}_1 ( (x_1'-c(q'))  \chi_{q'} \mathcal{F}_1^{-1} f_1) \cdot \Pi_{i=2}^{k}  \calE_i f_i  \|_{L^{\frac2{k-1}}(q)} \\
+ & \| x_1  \calE_1 \mathcal{F}_1 ( ( \nabla \varphi_1(\xi'_0) - \nabla \varphi_1(\xi')) \chi_{q'} \mathcal{F}_1^{-1} f_1) \cdot \Pi_{i=2}^{k}  \calE_i f_i  \|_{L^{\frac2{k-1}}(q)} \\
\leq &  A(R,\delta,\mu) \| (x'_1-c(q')) \chi_{q'} \mathcal{F}_1^{-1} f_1 \|_{L^2} \\
+ & A(R,\delta,\mu) \delta^{-1} R \| (\nabla \varphi_1(\xi'_0) - \nabla \varphi_1(\xi')) \chi_{q'} \mathcal{F}_1^{-1} f_1 \|_{L^2}  \Pi_{i=2}^{k} \| f_i \|_{L^2}  \\
\les & A(R,\delta,\mu) \left(  \| (x'_1-c(q')) \chi_{q'} \mathcal{F}_1^{-1} f_1 \|_{L^2} + R \| \chi_{q'} \mathcal{F}_1^{-1} f_1 \|_{L^2} \right)  \Pi_{i=2}^{k} \| f_i \|_{L^2} \\
\les & R A(R,\delta,\mu) \| \la \frac{x'_1-c(q')}{R} \ra \chi_{q'} \mathcal{F}_1^{-1} f_1 \|_{L^2}  \Pi_{i=2}^{k} \| f_i \|_{L^2}.
\end{split}
\]
We have used the following: \eqref{com} in justifying the equality between the terms on the second and fourth line, the induction hypothesis and the fact that inside $Q$ we have $|x_1| \lesssim \delta^{-1}R$ to justify the inequality in the sixth line. Note that it is in the above use of the induction estimate for $\calE_1 \mathcal{F}_1 ( (x'_1-c(q'))  \chi_{q'} \mathcal{F}_1^{-1} f_1)$ that we need to tolerate the relaxed support of $f_1$. $f_1$ is supported in the set $B_{\mu_1+10 (\delta^{-1} R)^{-1}}(M_1) \cap B_1(0,\delta)$, and this support is impacted by the convolution with $\mathcal{F}_1 ((x'_1-c(q'))  \chi_{q'})$ which is supported in $B_1(0,R^{-1})$; the sum set of the two supports is a subset of  $ B_{\mu_1+10 (\delta^{-1} R)^{-1}}(M_1)  \cap B_1(0,\delta)$;
thus we can invoke the induction hypothesis at scale $R$.

For any $q \in \calC(R) \cap Q$ and $x' \in \pi_{N_1}(q)$, it holds that 
$\la \frac{x'- c(q')-x_1 \nabla \varphi_1(\xi'_0)}R \ra \approx \la \frac{d(\pi_{N_1}(q),q')}R \ra$. 
This is justified by the fact that $|x_1| \lesssim \delta^{-1}R$ and $|\nabla \varphi_1(\xi'_0)| \leq \delta$, 
therefore the contribution of $|x_1 \nabla \varphi_1(\xi'_0)| \leq R$ is negligible. From this and the previous set of estimates,
we conclude that
\[
\begin{split}
&  \| \calE_1 \mathcal{F}_1 ( \chi_{q'} \mathcal{F}_1^{-1} f_1) \cdot \Pi_{i=2}^{k}  \calE_i f_i \|_{L^{\frac2{k-1}}(q)} \\
\les & A(R,\delta,\mu) \la  \frac{d(\pi_{N_1}q,q')}{R} \ra^{-1} \| \la \frac{x'_1-c(q')}{R} \ra \chi_{q'} \mathcal{F}_1^{-1} f_1 \|_{L^2}  \Pi_{i=2}^{k} \| f_i \|_{L^2}.
\end{split}
\]
We claim that we can extend the argument above to prove 
\begin{equation} \label{glocN}
\begin{split}
& \| \calE_1 \mathcal{F}_1 ( \chi_{q'} \mathcal{F}_1^{-1} f_1) \cdot \Pi_{i=2}^{k}  \calE_i f_i \|_{L^{\frac2{k-1}}(q)} \\
 \ls_N & A(R,\delta,\mu) \la  \frac{d(\pi_{N_1}q,q')}{R} \ra^{-N} \| \la \frac{x'_1-c(q')}{R} \ra^N \chi_{q'} \mathcal{F}_1^{-1} f_1 \|_{L^2}  \Pi_{i=2}^{k} \| f_i \|_{L^2},
\end{split}
\end{equation}
for all $N \geq 1$. In repeating the argument above, we need to start with the higher order terms:
\[
\| (x_1'- c(q')-x_1 \nabla \varphi_1(\xi'_0))^N \calE_1 \mathcal{F}_1 ( \chi_{q'}  \mathcal{F}_1^{-1} f_1) \cdot \Pi_{i=2}^{k}  \calE_i f_i \|_{L^{\frac2{k-1}}(q)}.
\]
If $N \geq 2$ some more, but harmless, terms appear due to the lack of commutativity between the symbols of the operators that are being used. 
We do this for $N=2$; the other cases are treated in a similar fashion. We use the following operator decomposition
\[
x_1'- c(q')-x_1 \nabla \varphi_1(\xi'_0) = A+B, \quad A = x_1'- c(q')-x_1 \nabla \varphi_1(\frac{D'}i), \quad B= x_1 ( \nabla \varphi_1(\xi'_0) - \nabla \varphi_1(\frac{D'}i)),
\]
based on which we have
\[
(x_1'- c(q')-x_1 \nabla \varphi_1(\xi'_0))^2 = (A+B)^2= A^2+ B^2+ 2BA + [A,B].  
\] 
Each component is treated as follows. The contribution of the term with $A^2$ is estimated as above and using \eqref{com}. 
The contribution of the term with $B^2$ is estimated as above and using the fact that 
\[
B^2= x_1^2 ( \nabla \varphi_1(\xi'_0) - \nabla \varphi_1(\frac{D'}i))^2,
\]
which is a consequence of the commutativity of $x_1$ and $\nabla \varphi_1(\xi'_0) - \nabla \varphi_1(\frac{D'}i)$. The contribution of the $BA$ term is estimated as follows:
\[
\begin{split}
& \| BA \calE_1 \mathcal{F}_1 ( \chi_{q'}  \mathcal{F}_1^{-1} f_1) \cdot \Pi_{i=2}^{k}  \calE_i f_i \|_{L^{\frac2{k-1}}(q)} \\
= & \| x_1 ( \nabla \varphi_1(\xi'_0) - \nabla \varphi_1(\frac{D'}i))  \calE_1 \mathcal{F}_1 ( (x_1'-c(q'))  \chi_{q'}  \mathcal{F}_1^{-1} f_1) \cdot \Pi_{i=2}^{k}  \calE_i f_i \|_{L^{\frac2{k-1}}(q)} \\
 \lesssim & \delta^{-1} R \|    \calE_1 ( \nabla \varphi_1(\xi'_0) - \nabla \varphi_1(\xi')) \mathcal{F}_1 ( (x_1'-c(q'))  \chi_{q'}  \mathcal{F}_1^{-1} f_1) \cdot \Pi_{i=2}^{k}  \calE_i f_i \|_{L^{\frac2{k-1}}(q)} \\
  \lesssim & \delta^{-1} R \|  ( \nabla \varphi_1(\xi'_0) - \nabla \varphi_1(\xi')) \mathcal{F}_1 ( (x_1'-c(q'))  \chi_{q'}  \mathcal{F}_1^{-1} f_1) \|_{L^2} 
  \cdot \Pi_{i=2}^{k}  \| f_i \|_{L^{2}} \\
   \lesssim & R \|  \mathcal{F}_1 ( (x_1'-c(q'))  \chi_{q'}  \mathcal{F}_1^{-1} f_1) \|_{L^2} 
  \cdot \Pi_{i=2}^{k}  \| f_i \|_{L^{2}} \\
 \lesssim & R \|  (x_1'-c(q'))  \chi_{q'}  \mathcal{F}_1^{-1} f_1 \|_{L^2}   \cdot \Pi_{i=2}^{k}  \| f_i \|_{L^{2}} \\
 \lesssim & R^2 \|  \la \frac{x_1'-c(q')}{R} \ra  \chi_{q'}  \mathcal{F}_1^{-1} f_1 \|_{L^2}   \cdot \Pi_{i=2}^{k}  \| f_i \|_{L^{2}}.
\end{split}
\]
Finally since $[A,B]= -\frac{x_1}i \Delta \varphi_1(\frac{D'}i)$, we estimate its contribution as follows
\[
\begin{split}
& \| [A,B] \calE_1 \mathcal{F}_1 ( \chi_{q'}  \mathcal{F}_1^{-1} f_1) \cdot \Pi_{i=2}^{k}  \calE_i f_i \|_{L^{\frac2{k-1}}(q)} \\
= & \| x_1 \Delta \varphi_1(\frac{D'}i)  \calE_1 \mathcal{F}_1 (  \chi_{q'}  \mathcal{F}_1^{-1} f_1) \cdot \Pi_{i=2}^{k}  \calE_i f_i \|_{L^{\frac2{k-1}}(q)} \\
 \lesssim & \delta^{-1} R \|    \calE_1 ( \Delta \varphi_1(\xi')) \mathcal{F}_1 (  \chi_{q'}  \mathcal{F}_1^{-1} f_1) \cdot \Pi_{i=2}^{k}  \calE_i f_i \|_{L^{\frac2{k-1}}(q)} \\
  \lesssim & \delta^{-1} R \|  \Delta \varphi_1(\xi') \mathcal{F}_1 (  \chi_{q'}  \mathcal{F}_1^{-1} f_1) \|_{L^2} 
  \cdot \Pi_{i=2}^{k}  \| f_i \|_{L^{2}} \\
   \lesssim & \delta^{-1} R \|  \mathcal{F}_1 ( \chi_{q'}  \mathcal{F}_1^{-1} f_1) \|_{L^2} 
  \cdot \Pi_{i=2}^{k}  \| f_i \|_{L^{2}} \\
 \lesssim & \delta^{-1} R \|  \chi_{q'}  \mathcal{F}_1^{-1} f_1 \|_{L^2}   \cdot \Pi_{i=2}^{k}  \| f_i \|_{L^{2}} \\
 \lesssim & R^2 \|   \chi_{q'}  \mathcal{F}_1^{-1} f_1 \|_{L^2}   \cdot \Pi_{i=2}^{k}  \| f_i \|_{L^{2}},
\end{split}
\]
where, in passing to the last line, we have used that $\delta^{-1} \ll R$. Based on all the estimates above, we obtain \eqref{glocN} for $N=2$. 
The above argument contains all the ingredients necessary to obtain \eqref{glocN} for $N \geq 3$. 

Using \eqref{triangle}, \eqref{pois} and \eqref{glocN}, we obtain
\[
\begin{split}
& \| \calE_1 f_1 \cdot \Pi_{i=2}^{k}  \calE_i f_i \|_{L^{\frac2{k-1}}(q)}^{\frac2{k-1}} \\
 \leq  & A(R,\delta,\mu)^{\frac2{k-1}}
\sum_{q' \in \calC \calH_1(R)} \| \calE_1 \mathcal{F}_1 ( \chi_{q'} \mathcal{F}_1^{-1} f_1) \cdot \Pi_{i=2}^{k}  \calE_i f_i \|_{L^{\frac2{k-1}}(q)}^{\frac2{k-1}} \\
 \ls_N & A(R,\delta,\mu)^{\frac2{k-1}} \left( \sum_{q' \in \calC \calH_{1}(R)} \la  \frac{d(\pi_{N_1}q,q')}{R} \ra^{-N \cdot \frac2{k-1}} \| \la \frac{x'-c(q')}{R} \ra^N \chi_{q'} \mathcal{F}_1^{-1} f_1 \|_{L^2}^\frac2{k-1} \right)  \Pi_{i=2}^{k} \| f_i \|^{\frac2{k-1}}_{L^2} \\
 \ls_N & A(R,\delta,\mu)^{\frac2{k-1}}  \Pi_{i=2}^{k} \| f_i \|^{\frac2{k-1}}_{L^2} \\
 & \cdot \left( \sum_{q' \in \calC \calH_{1}(R)} \la  \frac{d(\pi_{N_1}q,q')}{R} \ra^{-(2N-(k-1)^2)} \| \la \frac{x'-c(q')}{R} \ra^N \chi_{q'} \mathcal{F}_1^{-1} f_1 \|_{L^2}^2\right)^{\frac1{k-1}} .
\end{split}
\]
In justifying the last inequality, we have used the simple estimate for sequences 
\[
\| a_i \cdot b_i \|_{l^{\frac2{k-1}}_i} \lesssim \| a_i \|_{l^2_i}  \| b_i \|_{l^{\frac2{k-2}}_i}, 
\]
together with the straightforward estimate
\[
\| \la  \frac{d(\pi_{N_1}q,q')}{R} \ra^{-\frac{(k-1)^2}2} \|_{l^\frac{2}{k-2}_{q'}} \lesssim 1. 
\]
Note that the previous inequality is \eqref{INS} with the improvement for $f_1$.
By repeating the procedure for all other terms $f_2,..,f_{k}$ to conclude with \eqref{INS}. 

Using \eqref{INS} we invoke the discrete Loomis-Whitney inequality in \eqref{LWd} to conclude the argument. 
For $i=1,..k$, we define the functions $g_i: \mathcal{L}(\calH_i) \rightarrow \R$ by
\[
g_i(\bj)= \left( \sum_{q' \in \calC \calH_{i}(R)} \la  \frac{d(q(\bj),q')}{R} \ra^{-(N-2(k-1)^2)} \| \la \frac{x'-c(q')}{R} \ra^N \chi_{q'} \mathcal{F}_i^{-1} f_i \|^2_{L^2} \right)^\frac{1}2, \bj \in \mathcal{L}(\calH_i) .
\]
For $N$ large enough (depending only on $n$), we claim the following estimate:
\begin{equation} \label{keyloc}
\| g_i \|_{l^2l^\infty(\mathcal{L}(\calH_i') \times \mathcal{L}(\calH_i''))} \ls C_i(\mu_i,\delta,R) \| f_i \|_{L^2}, \quad i=1,..k,
\end{equation}
where $C_i(\mu_i,\delta,R)= \min(1,(R \mu_i + 10 \delta)^{\frac{c_i}2}$). From \eqref{LWd} we conclude 
\begin{equation} \label{claimA2}
A(\mu, \delta,\delta^{-1}R) \ls A(\mu,\delta,R) \Pi_{i=1}^k C_i(\mu_i,\delta,R).
\end{equation}
We will return to this estimate at the end of the proof and show how it is used to close the induction. 

For now, we continue with the argument for \eqref{keyloc}. We prove this claim in the case $i=1$, the other choices of $i$ being analogous;
in this context we ignore completely the variable along $n_1$. The variable in $\calH_1$ that was originally labeled by $x_1'$
will be denoted by $x$ and is further split as follows $x=(x',x'') \in \R^{n-c_1-1} \times \R^{c_1}$, with the dual Fourier variables 
being $\xi=(\xi',\xi'')$. $f_1$ is localized in $B_{\mu_1+10 \delta R^{-1}}(M_1)$ where $M_1$ is a $n-c_1-1$-dimensional manifold. As we discussed in Section \ref{sbr}, we can parametrize it as follows: $(\xi', \Phi(\xi'))$ with $\xi' \in \R^{n-c_1-1}$ and $\Phi(\xi')=(\tilde{\varphi}_{1}(\xi'), ..., \tilde{\varphi}_{c_1}(\xi'))$; this is true at least locally, but given that we work with $f_1$
supported in $B_1(0,20 \delta^2)$ with $\delta \ll 1$, we can obtain the parametrization in the whole support of $f_1$. 

The scope of what follows next is to "flatten" the manifold $M_1$, since this allows us to exploit the localization in an easier way. 
We write
\[
\begin{split}
\mathcal{F}_1^{-1} f_1 (x) & = \int e^{i x \cdot \xi} f_1(\xi) d\xi \\
& = \int e^{i (x' \cdot \xi'+ x'' \cdot (\Phi(\xi')+ \xi''))} f_1(\xi', \Phi(\xi') + \xi'') d\xi'  d\xi'' \\
& = \int \left( e^{i x'' \Phi(D_{x'})} \int e^{i (x' \cdot \xi'+ x'' \cdot \xi'')} h_1(\xi', \xi'') d\xi'  \right) d\xi'',
\end{split}
\]
where $h_1(\xi',\xi'')=  f_1(\xi', \Phi(\xi') + \xi'')$ and $e^{i x'' \Phi(D_{x'})}$ is the operator with symbol $e^{i x'' \cdot \Phi(\xi')}$, acting on $L^2_{x'}$.  
From its definition we see that $h_1$ is supported in a set where $|\xi'| \leq \delta + 10 R^{-1}$ and $\| \xi'' \|_{\infty} \leq \mu_{1}+10 \delta R^{-1}$. A straightforward computation shows that  $\| h_1 \|_{L^2}=\| f_1 \|_{L^2}$. 

We fix $x''$ and estimate the following quantity:
\[
\begin{split}
\|\mathcal{F}_1^{-1} f_1 (\cdot, x'')\|_{L^2_{x'}} & = \|\mathcal{F}_{1,x''}^{-1} f_1 (\cdot, x'')\|_{L^2_{\xi'}} \\
& = \| \int e^{ix''(\Phi(\xi')+\xi'')} h_1(\cdot, \xi'') d\xi'' \|_{L^2_{\xi'}} \\
& \leq \int \| e^{ix''(\Phi(\xi')+\xi'')} h_1(\cdot, \xi'')  \|_{L^2_{\xi'}} d\xi'' \\
& = \int \|  h_1(\cdot, \xi'')  \|_{L^2_{\xi'}} d\xi'' \\
& \leq (\Pi_{j=1}^{c_1} (\mu_1+10 \delta R^{-1}) )^\frac12 \| h_1 \|_{L^2} \\
& = (\mu_1+10 \delta R^{-1})^\frac{c_1}2 \| f_1 \|_{L^2}. 
\end{split}
\]
In the first line we have used the fact that $e^{i x''_1 \Phi_1(D_{x'})}$ is an isometry on $L^2_{x'}$ and this is the reason for skipping it in later computations.
In the last two lines, both $L^2$ norms, $\| h_1 \|_{L^2}$ and $\| f_1 \|_{L^2}$ are meant with respect to all variables, that is $L^2_{\xi',\xi''}$.

Then we integrate the above estimate with respect to the $x''$ variable in an interval of size $R$ (in all directions contained in $x''$) and localize within a cube $q$ of size $R$ to obtain
\[
\| \mathcal{F}_1^{-1} f_1 \|^2_{L^2(q)} \les  (R\mu_1+10 \delta)^\frac{c_1}2 \| f_1 \|_{L^2}
\]
The constant $C_1(\mu_1,\delta,R)$ obtained here has the correct numerology, but the estimate above misses some additional localization 
that is necessary in closing the argument for \eqref{keyloc}. 

Thus we need to provide a version of the above estimate that is spatially localized.  
For a cube $q' \in \calH_1'$, we let $\chi_{q'}(x')$ be a bump function at scale $R$
that is highly concentrated in $q'$, and consider 
\[
\begin{split}
\chi_{q'}(x') \mathcal{F}_1^{-1} f_1 (x) & = \int e^{i x \cdot \xi} (\hat \chi_{q'} \ast_{\xi'} f_1)(\xi) d\xi \\
& = \int e^{i x \cdot \xi} \int \hat \chi_{q'} (\eta') f_1(\xi'-\eta', \Phi(\xi'-\eta') + \xi'') d \eta' d\xi'  d\xi'' \\
& = \int  \hat \chi_{q'} (\eta') \int \left( e^{i x'' \Phi(D_{x'}-\eta')} \int e^{i (x' \cdot \xi'+ x'' \cdot \xi'')} h_1(\xi', \xi'') d\xi'  \right) d\xi'' d \eta',
\end{split}
\]
where $h_{1\eta'}(\xi',\xi'')=  f_1(\xi'-\eta', \Phi(\xi'-\eta') + \xi'')$ and $e^{i x'' \Phi(D_{x'}-\eta')}$ is the operator with symbol $e^{i x'' \cdot \Phi(\xi'-\eta')}$.
 For each $\eta'$, it follows from the definition that $h_1$ is supported in a set where $|\xi'-\eta'| \leq \delta + 10 R^{-1}$ and 
 $\| \xi'' \|_{\infty} \leq \mu_{1}+10 \delta R^{-1}$. 
Also it holds true that $\| h_{1\eta'} \|_{L^2}=\| f_1 \|_{L^2}$ for every $\eta'$.  Since $\hat \chi_{q'} \in L^1_{\eta'}$, it then follows that the above is an $L^1$-type superposition of operators that we have just analyzed, therefore we obtain the following estimate
\[
\| \chi_{q'} \mathcal{F}_1^{-1} f_1 (\cdot, x'')\|_{L^2_{x'}}  \les (\mu_1+10 \delta R^{-1})^\frac{c_1}2 \| f_1 \|_{L^2}.
\]
By itself this is not a new estimate; however we can repeat the process with $\chi_{q'}$ replaced by 
$(\frac{x-c(q')}{R})^N \chi_{q'}$, which has similar properties to those of $\chi_{q'}$ to obtain
\[
\| (\frac{x-c(q')}{R})^N \chi_{q'} \mathcal{F}_1^{-1} f_1 (\cdot, x'')\|_{L^2_{x'}}  \ls_N  (\mu_1+10 \delta R^{-1})^\frac{c_1}2 \| f_1 \|_{L^2}.
\]
Integrating this on an interval of length $R$ in the direction of $x''$ and restricting it to a cube $q \in \calH_1(R)$ gives
\begin{equation} \label{bass2}
\|  \chi_{q'} \mathcal{F}_1^{-1} f_1 \|_{L^2(q)}  \ls_N \la \frac{d(\pi_{\calH_1'}q,q')}R \ra^{-N} 
(R \mu_1+10 \delta)^\frac{c_1}2 \| f_1 \|_{L^2}.
\end{equation}
This estimate will take care of the $l^\infty$ part of the norm in \eqref{keyloc};
however it is not good enough to perform the $l^2$ summation efficiently since on the right-hand side above
 there is no memory of the fact that we are estimating the 
part of $\mathcal{F}_1^{-1} f_1$ that is highly localized in $q' \times \calH_1''$. 
This is easily fixed: if $\chi_{q'}(x')$ is compactly supported in, say, $2q'$ (the dilation of $q'$ by a factor of $2$ from its center),
we let $\tilde \chi_{q'}$ be a bump function adapted to $4q'$ with $ \chi_{q'} \cdot \tilde \chi_{q'}= \chi_{q'}$. Then running the above argument again, gives
\[
\|  \chi_{q'} \mathcal{F}_1^{-1} f_1 \|_{L^2(q)}  \ls_N \la \frac{d(\pi_{\calH_1'}q,q')}R \ra^{-N}  (R \mu_1+10 \delta)^\frac{c_1}2
\| \tilde \chi_{q'} \mathcal{F}_1^{-1} f_1 \|_{L^2}.
\]
A straightforward argument shows that \eqref{keyloc} follows from this last claim. 

We now return to \eqref{claimA2}
\[
A(\mu, \delta,\delta^{-1}R) \leq C A(\mu,\delta,R) \Pi_{i=1}^k C_i(\mu_i,\delta,R)
\]
and show how it implies the result in Theorem \ref{MB}. Recall that $C$  is independent of $\delta$, $R$ and $\mu$ and that
$C_i(\mu_i,\delta,R)= \min(1,(R \mu_j^i + 10 \delta))^{\frac{c_i}2}$. 
The localization in Theorem \ref{MB} and the one in Definition \ref{DefI} are not quite the same, but they match provided that 
$R \geq \max\{\mu_{1}^{-1},..,\mu_k^{-1}\}$; we first provide the argument for this case.

For any $R \geq \max\{\mu_{1}^{-1},..,\mu_k^{-1}\}$, we iterate \eqref{claimA2} to obtain
\[
A(\mu, \delta,R) \leq C^N A(\mu,\delta,\delta^{N} R) \Pi_{m=1}^{N} \Pi_{i=1}^k C_i(\mu_i,\delta,\delta^{m} R).
\]
We pick $N$ such that $\delta^{-1} \leq \delta^{N} R \leq \delta^{-2}$. We quantify the effect of some $\mu_i$ for some $i \in \{1,..,k \}$
 in the expression above as follows:
\[
\Pi_{m=1}^N \min(1,(\delta^{m} R \mu_i + 10 \delta)^{\frac{c_i}2}).
\] 
In the above we gain factors of $\delta$ for as long as $\delta^{m} R \mu_i \leq 10 \delta$; the accumulating powers of $10^\frac{c_i}2$ 
are added to the factor $C^N$ above. Thus it follows that
\begin{equation} \label{cuant}
\Pi_{m=1}^N \min(1,(\delta^{m} R \mu_i + 10 \delta)^{\frac{c_i}2}) \leq C_0^N \mu_i^\frac{c_i}2,
\end{equation}
where $C_0$ is independent of $\delta,R,\mu$; here we have used the fact that we have chosen $R$ large enough.
Thus we conclude with:
\[
A(\mu, \delta,R) \leq C^N A(\mu,\delta,\delta^{N} R) \Pi_{i=1}^k \mu_i^\frac{c_i}2,
\]
for any $R \geq \max\{\mu_{1}^{-1},..,\mu_k^{-1}\}$ and where $N$ is chosen such that $\delta^{-1} \leq \delta^{N} R \leq \delta^{-2}$.
From this we obtain
\[
A(\mu, \delta,R) \leq C^N \max_{r \leq \delta^{-2}} A(\mu,\delta,r) \Pi_{i=1}^k \mu_i^\frac{c_i}2. 
\]
From the uniform pointwise bound
\[
\| \Pi_{i=1}^{n+1} \calE_i f_i \|_{L^\infty} \ls \Pi_{i=1}^{n+1}  \| \calE_i f_i \|_{L^\infty} \ls \Pi_{i=1}^{n+1}  \| f_i \|_{L^2}
\]
it follows that $\max_{r \leq \delta^{-2}} A(\mu,\delta,r) \lesssim \delta^{-10}$. For  $R \in [\delta^{-N},\delta^{-N-1}]$, and $N$ large enough so that $R \geq \max\{\mu_{1}^{-1},..,\mu_k^{-1}\}$, the above implies
\[
A(\mu, \delta,R)  \leq C^N C(\delta) \leq R^\epsilon C(\delta)
\]
provided that $C^N \leq \delta^{-N\epsilon}$. Therefore choosing $\delta=C^{-\frac{1}{\epsilon}}$ leads to the desired result. 

We now consider the case $R \leq \max\{\mu_{1}^{-1},..,\mu_k^{-1}\}$. In this case, Definition \ref{DefI} requires localization properties
at larger scales since $R^{-1} \geq \min\{\mu_{1},..,\mu_k\}$; this is not a problem at all. However, when we run the argument above, 
we will not capture the full factor $\Pi_{i=1}^k \mu_i^\frac{c_i}2$, and this is reasonable given that we work with weaker localization properties. 

For simplicity, let us assume $\mu_1 \leq \mu_2 \leq  .... \leq \mu_k$ (which does not restrict the generality of the argument) and look at the particular case 
$\mu_1 \leq R^{-1} \leq \mu_2$ - the other cases are treated in a similar fashion. Then we run the above argument  and note that the factor $\Pi_{i=1}^k \mu_i^\frac{c_i}2$
needs to be replaced by
\[
R^{-\frac{c_1}2} \Pi_{i=2}^k \mu_i^\frac{c_i}2,
\]
which is a simple consequence of re-quantifying the estimate \eqref{cuant}. Thus, relatively to \eqref{Lf}, we are missing a factor of $(R \mu_1)^\frac{c_1}2$, which we needs to be recovered. As of now we obtain that the inequality
\begin{equation} \label{bass}
\| \Pi_{i=1}^{k} \calE_i f_i \|_{L^{\frac2{k-1}}(Q)} \leq C(\epsilon) R^\epsilon R^{-\frac{c_1}2} \Pi_{i=2}^k \mu_i^\frac{c_i}2 \Pi_{i=1}^{k} \| f_i \|_{L^2} 
\end{equation}
holds true for all cubes $Q \in \calC(R)$, under the assumption that $f_i$ is supported in $B_{\mu_i + 10R^{-1}}(M_i) \cap B(0,\delta), \forall i=1,..k$
As before we assume that $Q$ is centered at the origin. To improve this, we take advantage of the fact that, in the orginal estimate that we seek to establish, $f_1$ has better localization properties: it is supported in $B_{\mu_1}(M_1)$.  \eqref{INS} allows us to replace $\| f_1 \|_{L^2}$ in \eqref{bass} by
\[
\left( \sum_{q' \in \calC \calH_1(R)} \la \frac{d(\pi_{1} Q,q')}R \ra^{-(2N-n^2)} \| \la \frac{x_1'-c(q')}R \ra^{N} \chi_{q'} \mathcal{F}_1^{-1} f_1 \|_{L^2}^2 \right)^\frac12
\]
Morally this implies that, for the estimate at scale $R$, the contribution from $f_1$ that really matters is $\| \mathcal{F}_1^{-1} f_1 \|_{L^2(q)}$, where $q=\pi_1 Q \in \calC \calH_1(R)$, and the rest of the terms come with appropriate decay. On the other hand, a similar argument to the one used for \eqref{bass2} gives that for any $q \in \calC \calH_1(R)$
\[
\| \mathcal{F}_1^{-1} f_1 \|_{L^2(q)}  \ls  (R \mu_1)^\frac{c_1}2 \| f_1 \|_{L^2};
\]
here we use that $f_1$ has better localization properties. This brings in the correction we needed in \eqref{bass}; we leave the details as an exercise to the interested reader. 

 \subsection*{Acknowledgement}
Part of this work was supported by an NSF grant, DMS-1900603.
The author thanks Jonathan Bennett for in-depth help with the current literature on the subject. 
The author thanks Pavel Zorin-Kranich for bringing to our attention that some additional details would have been useful in our original work on the subject in \cite{Be1};
the corresponding details have been added in the current paper in Section \ref{MBP} in proving \eqref{INS} for the case $N=2$. 

\bibliographystyle{amsplain} \bibliography{HA-refs}

\end{document}